\documentclass[11pt]{amsart}

\usepackage{ae,aecompl}
\usepackage{esint}
\usepackage{graphicx}
\usepackage[all,cmtip]{xy}
\usepackage{amsmath, amscd}
\usepackage{booktabs}
\usepackage{amsthm}
\usepackage{amssymb}
\usepackage{amsfonts}
\usepackage{qsymbols}
\usepackage{latexsym}
\usepackage{mathrsfs}
\usepackage{cite}
\usepackage{color}
\usepackage{url}
\usepackage{enumerate}
\usepackage[english]{babel}

\usepackage{verbatim}
\usepackage[draft=false, colorlinks=true]{hyperref}

\usepackage{epsfig}
\usepackage{color}
\usepackage{epstopdf}

\usepackage{subfigure}
\usepackage{tikz}
\usepackage{tikz-3dplot}
\usetikzlibrary{arrows,automata}
\usepackage{varwidth}
\usepackage{tikz}
\usetikzlibrary{backgrounds,patterns,calc}
\usepackage{xparse}

\usetikzlibrary{calc, arrows.meta, decorations.pathreplacing}

\definecolor{ao}{rgb}{0, 0.5, 0}

\newcommand{\R}{\mathbb R}

\usepackage{tikz}
\usepackage{pgfplots}
\pgfplotsset{compat=1.13,
tick label style={color=white},
  label style={font=\small},
  legend style={font=\small}}

\usepackage{mathtools}

\allowdisplaybreaks

\newcommand{\mf}[1]{\mathfrak{#1}}

\newcommand {\C}{\mathbb C}

\newcommand {\ud}{\mathrm{d}}

\newcommand{\Ha}{\mathbb{H}^n}

\newcommand {\supp}{\mathrm{supp}}

\newcommand {\Z}{\mathbb Z}

\newcommand {\vanish}[1]{\relax}

\newtheorem{theorem}{Theorem}[section]
\newtheorem{lemma}[theorem]{Lemma}
\newtheorem{prop}[theorem]{Proposition}

\theoremstyle{definition}

\newtheorem{remark}[theorem]{Remark}

\numberwithin{equation}{section}
\protected\def\ignorethis#1\endignorethis{}
\let\endignorethis\relax

\title[Bilinear spherical maximal function]{Bilinear spherical maximal function on the Heisenberg group}


\author[Ghosh]{Abhishek Ghosh}
\address[Abhishek Ghosh]{Department of Mathematics, Indian Institute of Technology Madras, Chennai, 600036, India.}

\email{abhi@iitm.ac.in, abhi170791@gmail.com}

\author[Singh]{Rajesh K. Singh}
\address[Rajesh K. Singh]{Department of Mathematics, Gitam Institute of Science, GITAM University, Visakhapatnam- 530045, A.P., India}
\email{rsingh4@gitam.edu}



\begin{document}

\begin{abstract}
We introduce the bilinear Nevo–Thangavelu spherical means on the Heisenberg group $\mathbb{H}^n,$ and derive $L^{p_1}(\mathbb{H}^n) \times L^{p_2}(\mathbb{H}^n) \to L^{p}(\mathbb{H}^n)$ estimates for the single-scale bilinear averaging operators, the (full) bilinear Nevo–Thangavelu maximal operator and finally for the bilinear lacunary maximal operator on $\mathbb{H}^n; n \geq 2$. Our result for the full maximal operator is sharp. The principal tools in our analysis include newly developed estimates for single-scale bilinear averages, Hopf's maximal ergodic theorem, and a $T^*T$ argument adapted to this setting.
\end{abstract}
	
\maketitle	


\section{Introduction}
An important pursuit in modern real-variable harmonic analysis is the study of averages over lower-dimensional manifolds. This goes back to the pioneering work by Stein in \cite{SteinSph} where he proved that the spherical maximal function, defined by, \[
\mathcal{S}f(x):=\sup_{r>0}|S_{r}f(x)|,\quad   S_{r}f(x):=\int_{S^{n-1}}f(x-ry)\, d \sigma_{n-1}(y), \quad x \in \mathbb R^n,
\]
(where  $d \sigma_{n-1}$ is the normalized surface measure on $S^{n-1}$) is bounded on $L^p(\R^n)$ if and only if $p>n/(n-1)$ for $n\geq 3.$ In dimension two, Bourgain \cite{Bourgain} resolved the problem by showing that the circular maximal function is bounded from $L^p(\R^2)$ to itself for $p>2.$  The lacunary analogue, known as the lacunary spherical maximal operator, is defined as $$\mathcal{S}_{\text{lac}}f(x)=\sup_{\ell\in \Z} |S_{2^{\ell}}f(x)|,$$ and was studied by Coifman--Weiss \cite{CW} and by C. Calder\'on in \cite{C}, and it was proved that $\mathcal{S}_{\text{lac}}$ maps $L^p(\R^n)$ to itself for $1<p\leq \infty.$ A more general framework was developed by Duoandikoetxea and Rubio de Francia in \cite{DR} relating the boundedness of singular maximal operators $M_{\mu}f(x):=\sup_{\ell\in \Z}|f\ast \mu_{\ell}|$ with the Fourier decay of the underlying measure $\mu.$ These prompted a great deal of research in this direction and study of these maximal operators is still an active area of research. Recently, in  \cite{HickmanJFA}, the work of  Duoandikoetxea and Rubio de Francia is extended to the vast generality of homogeneous groups wherein the Fourier decay is appropriately replaced by the \textit{Curvature assumption} of the singular measure $\mu$ at hand. Needless to mention that this goes back to the foundational work of Ricci--Stein in \cite{RStein}. Motivated by these works, in this article we introduce bilinear spherical averages on the Heisenberg group $\Ha$ and study the associated lacunary and full bilinear spherical maximal operator on the Heisenberg group. To provide context
and clearly state our results, we introduce certain preliminaries at this stage.  

\medskip
\noindent\textit{The bilinear Nevo--Thangavelu maximal function.}
 Let $ \mathbb{H}^n:= \mathbb{C}^n \times \mathbb{R} $ denote the $(2n + 1)$-dimensional Heisenberg group, equipped with the group law defined for $ x=(z, t), y=(w, s) \in \mathbb{H}^n $ by  
\begin{equation}
(z,t) \cdot (w,s) := \big(z + w, \, t + s + \textstyle\frac{1}{2}  \Im(  z \cdot \bar{w}) \big).
\end{equation}
Further, $\delta_{r}(z,t):=(rz,r^2t)$ denotes the one-parameter family of dilations on $\mathbb{H}^n$ for every $r>0.$ The Haar measure on $\Ha$ is the Lebesgue measure $dzdt,$ and $\Ha$ is an homogeneous space equipped with the left-invariant Kor\'anyi norm $|(z, t)|=(\|z\|^4+|t|^2)^{1/4},$ $\|z\|$ is the Euclidean norm on $\C^n.$ Also, we have $|B(a, r)|=c_{n} r^{Q},$ where $Q=(2n+2),$  denotes the \textit{homogeneous dimension} of $\Ha,$ and $d=(2n+1)$ will denote the \textit{topological dimension} of $\Ha.$ For $r>0,$ and $f,g\in C^{\infty}_0(\mathbb{H}^n)$, we define the bilinear Nevo--Thangavelu averages $\mathfrak{S}_{r}$ as
\begin{equation}
    \mathfrak{S}_{r}(f,g)(x)=\int_{S^{4n-1}}f(x. \delta_{r}(z_1, 0)^{-1}) g(x. \delta_{r}(z_2, 0)^{-1})\, d\sigma_{4n-1}(z_1, z_2),\, x\in \Ha,
\end{equation}
$d\sigma_{4n-1}$ represents the rotation invariant normalized surface measure on $S^{4n-1},$ and here on wards we simply write $d\sigma$ if there is no confusion. The associated lacunary and full bilinear spherical maximal functions are, respectively, defined by
\begin{equation*}
    \mathfrak{M}_{\text{lac}}(f,g)(x)=\sup_{\ell \in \mathbb{Z}}|\mathfrak{S}_{2^{\ell}}(f,g)(x)|,\quad \mathfrak{M}_{\text{full}}(f,g)(x)=\sup_{r>0} |\mathfrak{S}_{r}(f,g)(x)|.
\end{equation*}
The primary objective of this article is to study the $L^{p_{1}}(\Ha)\times L^{p_{2}}(\Ha)\to L^{p}(\Ha)$ bounds for $\mathfrak{M}_{\text{lac}}$ and $\mathfrak{M}_{\text{full}}.$ We now state the first main result of this article concerning the boundedness properties of the single-scale averaging operators $\mathfrak{S}_{r}.$

\begin{theorem} Let $n \geq 2,$ and $d=2n+1.$ Suppose $(\frac{1}{p_1}, \frac{1}{p_2})$ is contained in the region $\mathcal{D}$, which consists of the open pentagon  with the corners $O=(0,0)$, $A=(0,1)$, $E=(\frac{d-1}{d},1)$, $F=(1,\frac{d-1}{d})$, $D=(1,0)$, together  with the line segments $[O,A], [A,E), (F,D], [D,O]$.  Let $0<p \leq \infty$ obeying the H\"older relation $\frac{1}{p} = \frac{1}{p_1} + \frac{1}{p_2}$. Then, we have 
\begin{equation*}
   \|  \mathfrak{S}_{r}(f,g)\|_{L^{p}(\Ha)} \leq C \|f\|_{L^{p_{1}}(\Ha)} \|g\|_{L^{p_{2}}(\Ha)},
\end{equation*}  
uniformly for all $r>0.$
\label{single-scale-thm}
\end{theorem}

To provide the right context to the above result, we note that multilinear convolution operators of the form $$T_{n}(f_1, f_2,\cdots, f_{n})(x)=\int_{S^{n-1}}\prod_{i=1}^{n}f_{i}(x-z_{i})\, d\sigma(z_1, \cdots, z_{n}),\, x\in \R,$$
were studied by Oberlin \cite{Obe88}. We also note the interesting result by Shrivastava and Shuin in \cite{SS21} where they studied the operator $T_{n}$ for Banach range of indices. Recently, in \cite{IoseJGA} the authors studied Euclidean bilinear  spherical means
$$\mathcal{A}_{r}(f, g)(x)=\int_{S^{2n-1}}|f(x-ry)g(x-rz)|\, d\sigma_{2n-1}(y, z),$$
and proved that $\mathcal{A}_{r}$ maps $L^1(\R^n)\times L^1(\R^n)$ to $L^s(\R^n)$ with $s\in [1/2, 1],$ they also proved similar results for the triangle averaging operators. Our result Theorem~\ref{single-scale-thm} extends these to the Heisenberg group.

When it comes to the Euclidean bilinear spherical maximal operator $$\mathcal{M}(f, g)(x):=\sup_{r>0}|\mathcal{A}_{r}(f, g)(x)|,$$ there are many serious developments over the last couple of years. To start with, in \cite{IoseMRL} the authors proved that
$\mathcal{M}: L^2(\R^n)\times L^2(\R^n)$ to $L^2(\R^n),$ and subsequently, Barrionuevo \textit{et al} \cite{GraMRL} proved that $\mathcal{M}$ is bounded from $L^p(\R^n)\times L^q(\R^n)\to L^r(\R^n),$ where $\frac{1}{r}=\frac{1}{p}+\frac{1}{q}$ and $(\frac{1}{p}, \frac{1}{q})$ lies in the open quadrilateral with vertices $(0, 0),$ $(1, 0),$ and $(0, 1)$ and $(\frac{2n-10}{2n-5}, \frac{2n-10}{2n-5}),$ and was further improved in \cite{GHH21}, and in \cite{HHY20}. The main idea in \cite{GraMRL} was to use wavelet decomposition, however, very recently, Jeong and Lee obtained the sharp range of exponents proving that for $n\geq 2, 1\leq p, q\leq \infty$ and $0<r\leq\infty,$ the maximal function  $\mathcal{M}$ maps $L^p(\R^n)\times L^q(\R^n)$ to $L^r(\R^n),$ if and only if $r>\frac{n}{2n-1},$ $\frac{1}{r}=\frac{1}{p}+\frac{1}{q};$ except the points $(p, q, r)=(1, \infty, 1)\, \text{or}\, (\infty, 1, 1)$ where they obtained appropriate weak-type estimates. They opened a new paradigm for studying bilinear averages using the slicing method, we recall it here as it will be useful for our purpose later. The authors in \cite{LeeJFA} showed that for any continuous function $G$ on $\R^{2n}, n\geq 2,$ one can write
\begin{align}
\label{slicing-formula}
&\int_{S^{2n-1}} G(z_1, z_2)\,d\sigma_{2n-1}(z_1, z_2)\\
\nonumber&= \int_{B^{n}(0, 1)}\int_{S^{n-1}} G(z_1, \sqrt{1-\|z_1\|^2} z_2)\, d\sigma_{n-1}(z_2)(1-\|z_1\|^2)^{(n-2)/2}\, dz_1.    
\end{align}
Using this they controlled the maximal operator $\mathcal{M}(f, g)(x)$ by 
\begin{align}
\label{lee}
\min\{M_{\text{HL}}f(x)\mathcal{S}g(x), M_{\text{HL}}g(x)\mathcal{S}f(x)\}    
\end{align}
 and then an application of H\"older's inequality concludes their proof. For $n=1,$ the bilinear maximal operator $\mathcal{M}$ was studied by Christ and Zhou \cite{CZ24}, and by Dosidis and Ramos \cite{DR24}, we also refer \cite{BFO23} for sparse domination and weighted estimates. Now we state our main result for the full maximal operator $\mathfrak{M}_{\text{full}}.$

\begin{theorem}
\label{full}
Let $n\geq 2.$ Suppose $\big(\frac{1}{p_1}, \frac{1}{p_2}\big)$ is contained in the region $\mathcal{P}$, which consists of the open pentagon $\mathcal{P}$ with the corners $O=(0,0)$, $A=(0,1)$, $B=(\frac{d-2}{d-1},1)$, $C=(1,\frac{d-2}{d-1})$ and $D=(1,0)$, together with half-open line segments $[O,A)$ and $[O,D),$   see Figure \ref{fig: full max region}.  Let $0<p \leq \infty$ obeying the H\"older relation $\frac{1}{p} = \frac{1}{p_1} + \frac{1}{p_2}$. Then, we have 
\begin{equation}
   \|  \mathfrak{M}_{\mathrm{full}}(f,g)\|_{L^{p}(\Ha)} \leq C \|f\|_{L^{p_{1}}(\Ha)} \|g\|_{L^{p_{2}}(\Ha)}.
   \label{full:max}
\end{equation}
\end{theorem}
A few words on the proof of Theorem~\ref{full} are in order. Here, we observe that one can use the slicing argument, but the presence of the Heisenberg group law and higher co-dimension makes it more complicated than the Euclidean counterpart and the appropriate replacement of \eqref{lee} involves the product of the Nevo--Thangavelu maximal operator $M_{S},$ defined by \eqref{NT}, and an ergodic maximal operator $\Lambda,$ defined by \eqref{ergodic-max}. This ergodic maximal operator made its appearance in earlier works of Nevo and Thangavelu in \cite{NeT} and of Narayanan--Thangavelu in \cite{NaT}. Another important feature is that the Theorem~\ref{full} is sharp. This is the content of the following result.
\begin{prop}
\label{nece:loc}
Let $n\geq 2.$ Let $1\leq p_1, p_2\leq \infty$ and $0<p<\infty.$ If we have \[\|\mathfrak{M}_{\mathrm{full}}(f, g)\|_{L^p(\Ha)}\leq C\|f\|_{L^{p_1}(\Ha)}\|g\|_{L^{p_2}(\Ha)}\] then we must have
\begin{align}
\label{Nece:2}
\frac{1}{p_1}+\frac{1}{p_2}\leq \frac{2d-3}{d}+\frac{1}{pd}.
\end{align}
In fact, we prove a stronger statement by showing that condition (\ref{Nece:2}) is also necessary for the boundedness of the local maximal function $\mathfrak{M}_{\mathrm{loc}},$ defined by 
$$\mathfrak{M}_{\mathrm{loc}}(f, g)(x) :=\sup_{1\leq r\leq 2}|\mathfrak{S}_{r}(f, g)(x)|.$$
\end{prop}

\begin{remark} Considering the scaling relation:
    $$\mf{S}_{rR}(f,g)= \mf{S}_{r} (f \circ \delta_{R}, g \circ \delta_{R}) \circ \delta_{1/R},$$
we get the H\"older's condition $\frac{1}{p}=\frac{1}{p_1}+\frac{1}{p_2}$ is necessary for \eqref{full:max}. However, for $\mf{M}_{\text{loc}}$  we further obtain a necessary condition $\frac{2}{p_1}+\frac{2}{p_2}\leq 1+\frac{d}{p}$ for the boundedness of $\mf{M}_{\text{loc}}$ from $L^{p_{1}}(\Ha)\times L^{p_2}(\Ha)\to L^p(\Ha),$ see Proposition~\ref{Nece:prop:2}.
\end{remark}

\begin{figure}[ht]
\centering
\begin{tikzpicture}[scale=5] 
    \draw[->] (-0.1,0) -- (1.2,0) node[right] {$\frac{1}{p_{1}}$};
    \draw[->] (0,-0.1) -- (0,1.2) node[above] {$\frac{1}{p_{2}}$};
    
    \draw[densely dotted] (0,1) -- (1,1) -- (1,0);
    \node at (1,-0.1) {$1$};
    \node at (-0.1,1) {$1$};

    \coordinate (O) at (0,0);
    \coordinate (A) at (0, 1);
    \coordinate (B) at (0.7,1);
    \coordinate (B proj) at (0.7,0);
    \coordinate (C) at (1, 0.7); 
    \coordinate (C proj) at (0, 0.7); 
    \coordinate (D) at (1,0); 

    \coordinate (A') at (0, 0.99);
    \coordinate (B') at (0.7,0.99);
    \coordinate (C') at (0.99, 0.7); 
    \coordinate (D') at (0.99,0);

    \fill[gray!17] (O) -- (A') -- (B') -- (C') -- (D') -- cycle;
    \fill[gray!17] (O) -- (A) -- (B) -- (B proj) -- (O) -- cycle;
     \fill[gray!17] (O) -- (C proj) -- (C) -- (D) --  (O) -- cycle;


\draw[thin, gray! 40] (0.7,0) -- (0.7,1);
\draw[thin, gray! 40] (0,0.7) -- (1,0.7);

    \draw[thick] (O) -- (0,0.99) node[midway, above left, black, font=\tiny] {}; 
    \draw[densely dotted, thick] (0.01, 1) -- (0.69,1) node[midway, right, black, font=\tiny] {};
    \draw[densely dotted, thick] (B) -- (C) node[midway, below right, black, font=\tiny] {};
    \draw[densely dotted, thick] (1, 0.69) -- (1,0.01) node[midway, right, black, font=\tiny] {};
     \draw[thick] (0.99,0) -- (O) node[midway, right, black, font=\tiny] {};
    
     \node[above right] at (A) {$A$};
     \node[above] at (B) {$B=\textstyle (\frac{d-2}{d-1}, 1)$};
     \node[right] at (C) {$C=\textstyle (1,\frac{d-2}{d-1})$};
      \node[above right] at (D) {$D$};
    \node[below left] at (O) {$O$};
    
    \fill (O) circle (0.4pt);
    \draw (A) circle (0.4pt);
    \draw (B) circle (0.4pt);
    \draw (C) circle (0.4pt);
    \draw (D) circle (0.4pt);
    
    \node at (0.50, 0.50) {$\mathcal{P}$};
\end{tikzpicture}
\caption{$\mathfrak{M}_{\mathrm{full}}$-boundedness region $\mathcal{P}$ which is the union of open pentagon with vertices $O,A,B,C,D$ and the line segments $[O,A)$ and $[O,D)$.}
\label{fig: full max region}
\end{figure}
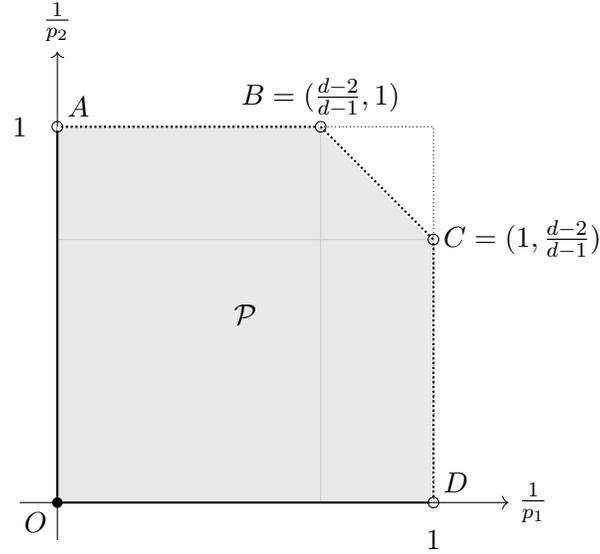
Building further upon the works of Jeong and Lee \cite{LeeJFA}, Borges and Foster \cite{BF24} studied the lacunary Euclidean bilinear maximal operator $$\mathcal{M}_{\text{lac}}(f, g)(x):=\sup\limits_{\ell\in \Z} \mathcal{A}_{2^{\ell}}(f, g)(x),$$ and proved that $\mathcal{M}_{\text{lac}}(f, g): L^p(\R^n)\times L^q(\R^n) \rightarrow L^r(\R^n)$ for $1<p, q\leq \infty$ and $0<r\leq \infty$ satisfying the H\"older relation on $n\geq 2$. They also studied the lacunary bilinear triangle averaging operator in \cite{BF24}.  Now we state our main results for the lacunary maximal operator $\mathfrak{M}_{\text{lac}}.$


\begin{figure}[ht]
\centering
\begin{tikzpicture}[scale=5] 
    \draw[->] (-0.1,0) -- (1.2,0) node[right] {$\frac{1}{p_{1}}$};
    \draw[->] (0,-0.1) -- (0,1.2) node[above] {$\frac{1}{p_{2}}$};
    
    \draw[densely dotted] (0,1) -- (1,1) -- (1,0);
    \node at (1,-0.1) {$1$};
    \node at (-0.1,1) {$1$};

    \coordinate (O) at (0,0);
    \coordinate (A) at (0, 1);
    \coordinate (B) at (0.7,1);
    \coordinate (C) at (1, 0.7); 
    \coordinate (Bl) at (0.85,1);
    \coordinate (Bs) at (0.84,1);
    \coordinate (Cl) at (1, 0.85); 
    \coordinate (Cs) at (1, 0.84); 

    \coordinate (D) at (1,0); 

    \coordinate (A') at (0, 0.99);
    \coordinate (B') at (0.7,0.99);
    \coordinate (C') at (0.99, 0.7); 
     \coordinate (Bl') at (0.85,0.99);
    \coordinate (Cl') at (0.99, 0.85); 
    \coordinate (D') at (0.99,0);

    \fill[gray!17] (O) -- (A) -- (Bs) -- (Cs) -- (D) -- cycle;


\draw[thin, gray! 40] (0.7,0) -- (0.7,1);
\draw[thin, gray! 40] (0,0.7) -- (1,0.7);
\draw[thin, gray! 40] (0.85,0) -- (0.85,1);
\draw[thin, gray! 40] (0,0.85) -- (1,0.85);

    \draw[thick] (O) -- (A) node[midway, above left, black, font=\tiny] {}; 
    \draw[densely dotted, thick] (A) -- (B) node[midway, right, black, font=\tiny] {};
    \draw[densely dotted, thick] (B) -- (Bs) node[midway, right, black, font=\tiny] {};
    \draw[densely dotted, thick, gray! 50] (B) -- (C) node[midway, below right, black, font=\tiny] {};
        \draw[densely dotted, thick] (0.86,0.99) -- (0.99,0.86) node[midway, below right, black, font=\tiny] {};
    \draw[densely dotted, thick] (Cs) -- (C) node[midway, right, black, font=\tiny] {};
    \draw[densely dotted, thick] (C) -- (D) node[midway, right, black, font=\tiny] {};
     \draw[thick] (D) -- (O) node[midway, right, black, font=\tiny] {};
    
     \node[above right] at (A) {$A$};
     \node[circle] at (0.6,1.06) { \textcolor{gray}{ \scalebox{0.7}{$  B= (\frac{d-2}{d-1}, 1)$} } };
\node[circle] at (1, 1.07) {$E=\textstyle (\frac{d-1}{d}, 1)$};
     \node[circle] at (1.25,0.83) {$F=\textstyle (1,\frac{d-1}{d})$};
     \node[circle] at (1.19,0.65) {\textcolor{gray}{  \scalebox{0.7}{$  C=(1,\frac{d-2}{d-1})$}    }  };
      \node[above right] at (D) {$D$};
    \node[below left] at (O) {$O$};
    
    \fill (O) circle (0.4pt);
     \draw (A) circle (0.4pt);
    \draw (B) circle (0.4pt);
    \draw (C) circle (0.4pt);
    \draw (Bl) circle (0.4pt);
    \draw (Cl) circle (0.4pt);
    \draw (D) circle (0.4pt);
    
    \node at (0.50, 0.50) {$\mathcal{R}$};
\end{tikzpicture}
\caption{$\mathfrak{M}_{\mathrm{lac}}$-boundedness region $\mathcal{R}$ which is the union of open pentagon with vertices $O,A,E,F,D$ and line segments $[O,A)$ and $[O,D)$.
}
\label{fig: lac max region}
\end{figure}
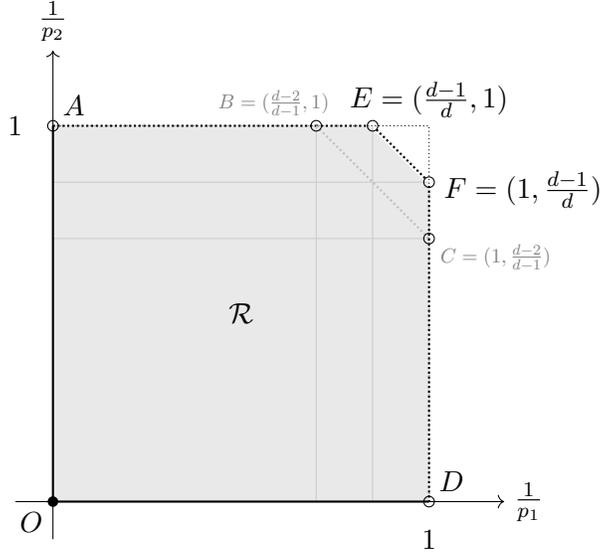

\begin{theorem}\label{Horizontal lac result}
Let $n \geq 2$. Suppose $(\frac{1}{p_1}, \frac{1}{p_2})$ is contained in  the region $\mathcal{R}$, 
which consists of the open pentagon with the corners $O=(0,0)$, $A=(0,1)$, $E=(\frac{d-1}{d},1)$, $F=(1,\frac{d-1}{d})$ and $D=(1,0)$, 
 together with the line segments $[O,A), (D,O]$,  see Figure \ref{fig: lac max region}. Let $0<p \leq \infty$ obeying the H\"older relation $\frac{1}{p} = \frac{1}{p_1} + \frac{1}{p_2}$. Then, we have
\begin{equation*}
  \|  \mathfrak{M}_{\mathrm{lac}}(f,g)\|_{L^{p}(\Ha)} \leq C \|f\|_{L^{p_{1}}(\Ha)} \|g\|_{L^{p_{2}}(\Ha)}.
\end{equation*}
\end{theorem}
Theorem~\ref{single-scale-thm}, slicing argument, and the Littlewood-Paley decomposition given in Proposition~\ref{LP} play quintessential role in the proof Theorem~\ref{Horizontal lac result}. The primary difference with the work \cite{BF24} lies on the fact that the Fourier transform on $\Ha$ becomes operator-valued and makes the use of the Fourier decay of the surface measure much more challenging. To circumvent this issue, we resort to the approach taken by \cite{HickmanJFA} and a suitable use of the $T^*T$ argument combined with the \textit{Curvature assumption} yields the desired $L^2(\Ha)\times L^2(\Ha)\to L^1(\Ha)$ with decay for the ``frequency-localised" pieces of $\mathfrak{S}.$ We believe one can also follow the approach by M\"uller--Seeger \cite{MS}  and use FIO based techniques, however, we think our approach maybe more suitable for extending this results to general homogeneous groups.

To have a proper historical perspective, we say a few words on the linear Nevo--Thangavelu spherical means. Let $\mu$ be the normalized surface measure on the horizontal sphere $S^{2n-1} \times \{0\} \subseteq \Ha$. For $r>0,$ the dilate of $\mu,$ $\mu_{r}$ is defined as $\langle f, \mu_{r}\rangle=\langle f\circ \delta_{r}, \mu\rangle,$ for Schwartz class functions. The linear Nevo--Thangavelu spherical means are the averages over horizontal spheres, defined as  
\begin{align}
A_{r} f(x):=f*\mu_{r}(x)=\int_{S^{2n-1}} { \textstyle f\big(z-r u, t- {\textstyle\frac{1}{2} } r  \Im(z \cdot \bar{u})\big) \,d \sigma_{2n-1}(u),\,x=(z, t)\in\mathbb{H}^n,    }
\end{align}
and for $r=1,$ we simply denote the averaging operator $A_1,$ by $A.$ Finally, the Nevo--Thangavelu spherical maximal function, for Schwartz class functions, is defined as 
\begin{align}
\label{NT}
M_{S}f(x)=\sup_{r>0}|A_{r}f(x)|,\quad x \in \mathbb{H}^n.    
\end{align}
 The authors in \cite{NeT} introduced the operator $M_{S}$ as an analogue of the Stein's spherical maximal (\cite{SteinSph}) function on $\Ha$ and proved the ergodicity of the spherical means $A_{r}$ by proving that $M_{S}$ is bounded on $L^p(\mathbb{H}^n)$ for $p > \frac{2n-1}{2n-2}$ with $n\geq 2$. Subsequently, it was independently proved by M\"uller--Seeger in \cite{MS} and by Narayanan--Thangavelu in \cite{NaT} that $M_{S}$ is bounded on $L^p(\mathbb{H}^n)$ if and only if $p > \frac{2n}{2n-1}$ and $n\geq 2$. On $\mathbb{H},$ we mention the important work \cite{BeltranRad} where the authors obtained Lebesgue space estimates of Nevo--Thangavelu maximal function on $\mathbb{H}$ restricted to the Heisenberg radial functions. In recent times also the operator $M_{S}$ has attracted a great deal of research, for instance, in \cite{BHRT} the authors studied the sparse operator bounds, weighted estimates for $M_{S},$ they also obtained the $L^p$ improving estimates for spherical means $A_{r}$ and thus built the analogue of the influential work of Lacey in \cite{Lacey} in the Heisenberg group. These were further improved and extended to the more general context of M\'etivier groups in \cite{Roos} using oscillatory integrals and Carleson-Sj\"olin estimates. In a recent breakthrough in \cite{SeegerAdv}, the authors extended the M\"uller--Seeger result on any two-step nilpotent Lie group. In view of the above, we believe our results are a timely addition to the rich literature in this direction. Lastly, we mention that there is another type of maximal operator on $\Ha$, known as the Kor\'anyi maximal operator studied in \cite{PT, Rajula} and bilinear analogue of them will be addressed in a future work. 

\medskip
\noindent\textit{Key ideas in proofs.} We close this section by key ideas in the proofs of our results. 
\begin{itemize}
\item
To start with, we obtain a description of the H\"older boundedness of single-scale operators for all indices $(\frac{1}{p_1}, \frac{1}{p_2})$ in the region $\mathcal{D}.$ Such an estimate will play a crucial role in the proof of the lacunary maximal function. This entails a lot new features not present in the Euclidean counterpart proved in \cite{IoseJGA}. For instance, the higher co-dimension of our operators and the non-abelian structure makes the use of slicing argument not directly applicable in our context. In Proposition~\ref{single:improv}, we compute the density of the pushforward measure using certain change of variables which are particularly tailored for this purpose. This is a major observations in this present work and is the content of our Section~\ref{sec:single}.
\item
Another key element that arises in the proof of Theorem~\ref{full} is the appearance of uniform spherical averages.
These are controlled by the ergodic maximal function associated to Hopf's theorem, unlike the Euclidean analogue, see Section~\ref{sec:full}.  Moreover, we consider Knapp-type examples to show that the Theorem~\ref{full} is sharp, and is available in Section~\ref{sec:sharp}.  
\item
After the single scale estimates in Section~\ref{sec:single} are available, we show that the ``dyadically localised" single-scale operators satisfy appropriate decay estimate. To avoid the use of the operator-valued group Fourier transform, we employ a $T^*T$ argument and this makes the proof of such an estimate much more challenging than the Euclidean counterpart. Finally, a delicate application of an argument by Christ combines ``dyadically-localised" pieces to obtain the final estimate for the lacunary maximal function. Here also our analysis involves a two-parameter family of Littlewood-Paley projections unlike the Euclidean case where one can easily construct Littlewood-Paley functions $\{\phi_{k}\}$ and $\{\phi'_{k}\}$ such that $\sum_{k}\phi_{k}\ast \phi'_{k}=\delta$ using the Fourier transform, a fact heavily used in such analysis, see \cite{BF24, HHY20}. 

Finally, we conclude by noting that the validity of the range of exponents for the boundedness of the lacunary maximal function, i.e., the region $\mathcal{R},$  is restricted by the
validity of the single-scale estimate Theorem~\ref{single-scale-thm}. The problem of obtaining the sharp range of exponents for single-scale operators will be handled in a future work.
\end{itemize}

\subsection*{Notation:} We use the notation $[A,B)\, \text{and}\, (A, B)$ to denote half-open (excluding $B$) and open line-segment joining the two points $A\, \text{and}\, B$ respectively. Also, the $L^{p}(\Ha)$ norm of a function $f$ will be denoted by $\|f\|_{L^p(\Ha)}$ or by $\|f\|_{p}.$ For $x\in \Ha$ and a function $g,$ the left translate is denoted by $\tau_{x}g(y):=g(x.y)$ for all $y\in \Ha.$ The notation $A\lesssim B$ represents that $A\leq  C B,$ for some implied constant $C$ which may depend on $n$, and $A\simeq B$  indicates $A\lesssim B$ and $B \lesssim A.$

\section{Preliminaries}
In this section, we collect some preliminaries which will be useful later. As we already defined, $\Ha:=\mathbb{C}^n\times\mathbb{R}$ is the $(2n+1)$-dimensional two-step nilpotent Lie group, equipped with group law
$$(z,t) \cdot (w,s) := \big(z + w, \, t + s + \textstyle\frac{1}{2}  \Im(  z \cdot \bar{w}) \big).$$

The Lebesgue measure $dzdt$ on $\mathbb{C}^n\times\mathbb{R}$ serves as the Haar measure on $\Ha$. As already mentioned, there is a family of parabolic dilations defined by $\delta_{r}(z, t):=(r z, r^2 t)$, and the Kor\'anyi norm is homogeneous of degree 1 with respect to this family of dilations, that is, $|\delta_{r}(z, t)|= r \,  |(z, t)|$. The convolution of $f$ with $g$ on $\Ha$ is defined by
\begin{equation*}
   f * g \, (x) = \int_{\Ha}  f(x y^{-1}) g(y) dy, \ \ \ x \in \Ha.
\end{equation*}
Also, for $r>0$ we denote \begin{equation*}
    f_r(x) := r^{-Q}f \circ \delta_{r^{-1}}(x);
\end{equation*}
where $Q=2n+2$ is the homogeneous dimension of $\Ha.$ For $f\in L^1_{\mathrm{loc}}(\Ha)$, the centered Hardy-Littlewood maximal function $M_{\text{HL}}f$ is defined by 
\begin{align}
M_{\text{HL}}f(x):=\sup_{r>0}\frac{1}{|B(x, r)|}\int_{B(x, r)} |f(y)|\, dy, \quad x\in \Ha, 
\end{align}
where the supremum is over all balls $B(x, r),$ where $B(x, r)$ denotes the ball of radius $r$ centered at $x$ under the left invariant metric $d_{L}(x, y)= |x^{-1}\cdot y|.$ The standard $L^p$ boundedness for $1<p\leq \infty$ and the weak $(1,1)$ boundedness of $M_{\text{HL}}$ follow as $\Ha$ is a space of homogeneous type. 
\medskip

\noindent{\textit{Littlewood-Paley theory.}} Littlewood-Plaey decomposition are well-known for homogeneous groups, we recall it here specialized to the Heisenberg group $\Ha.$ Following \cite{HickmanJFA}, consider a function $\phi\in C_{c}^{\infty}(\Ha)$ with mean $1,$ that is, $\int \phi=1.$ Moreover, for $k\in \Z,$ the functions $\psi_{2^k}$ are defined as
$$\psi_{2^{k}}(x):=-\int_{2^{k}}^{2^{k+1}}\frac{\partial \phi_{r}(x)}{\partial r}\, dr.$$
Then the family $\{\psi_{2^k}\}_{k\in \Z}$ forms
a Littlewood-Paley decomposition in the sense:
\begin{prop}[\cite{HickmanJFA}]
\label{LP}
Let $1<p<\infty$ and $f\in C_{c}^{\infty}(\Ha).$ Then we have  $$f=\sum_{k}f\ast \psi_{2^k},$$
where the convergence is in $L^p(\Ha)$ norm. 
\end{prop}

\noindent{\textit{Birkhoff averages.}} As mentioned earlier, an important operator that arises very naturally are the following uniform averages of linear Nevo--Thangavelu means. More precisely, let $\mu=\sigma_{2n-1}\ast \delta_{0}$ denotes the measure on the horizontal sphere $S^{2n-1}\times \{0\}$, then 
$$A_{s}f(x):=f\ast \mu_{s}(x)=\int_{S^{2n-1}} f(x. \delta_{s}(z, 0)^{-1})\, d\sigma_{2n-1}(z).$$
Consider the maximal operator 
\begin{equation}
\label{ergodic-max}
 \Lambda(g)(x):= \sup_{r>0}|\Lambda_{r}g(x)|:= \sup_{r>0} \big|g \ast {\Big( \frac{1}{r} \int_{0}^{r} \mu_{s} \, ds  \Big)} \big|.
\end{equation}
Realizing the uniform averages $\Lambda_{r}$ as Birkhoff averages over the group of reals, the following result follows from the Hopf's abstract maximal ergodic theorem.
\begin{prop}[\cite{SS} or Proposition 3.6.1, \cite{ThangBook}] \label{Birkhoff avg} Let $1<q< \infty$. Then $\Lambda$ defines a bounded operator on $L^q(\Ha)$ to itself, that is,
\begin{equation*}
  \|\Lambda g\|_{L^q(\Ha)} \leq C \| g \|_{L^q(\Ha)}.
\end{equation*} 
\end{prop}
This operator played a crucial role in the works \cite{NeT, NaT}. As the reader will observe in the sequel, $\Lambda$ also plays a pervasive role in many of our estimates.

\section{Single scale estimates}
\label{sec:single}
This section is dedicated in proving the boundedness of single-scale operators $\mathfrak{S}_{r},$ Theorem~\ref{single-scale-thm}. We begin with the following proposition. Besides being a primary ingredient in the proof of boundedness of the lacunary maximal function, the following result is of independent interest and involves delicate adaptations pertaining to the Heisenberg group. For brevity, we denote $\mf{S}_{1}$ as $\mf{S}.$
\begin{prop}
\label{single:improv}
Let $\frac{d}{d-1} < p_2<\infty.$ Then the operator $\mathfrak{S}$ is bounded from $L^1(\Ha)\times L^{p_2}(\Ha)\rightarrow L^{1}(\Ha).$ Similarly, for $\frac{d}{d-1} <  p_1<\infty,$ $\mathfrak{S}$ is bounded from $L^{p_1}(\Ha) \times L^1(\Ha)\rightarrow L^{1}(\Ha).$
\end{prop}

\begin{proof}
We shall only prove the first estimate as the second one follows by similar arguments. As $\mathfrak{S}$ is a positive operator, we can take $f, g$ to be non-negative functions. Performing a change of variables, we can write the $L^1$ norm of $\mathfrak{S}(f,g)$ as
\begin{align}
\nonumber&\int_{\Ha}  \mathfrak{S}(f,g)(x)\, dx\\
\nonumber& = \int_{\Ha}\int_{S^{4n-1}}f(x. (z_1, 0)^{-1}) g(x. (z_2, 0)^{-1})\, d\sigma(z_1, z_2)\,dx\\
\nonumber&=\int_{\Ha} f(x) \int_{S^{4n-1}} g(x. (z_1, 0).(z_2, 0)^{-1})\, d\sigma(z_1, z_2)\, dx=:\int_{\Ha} f(x)\, \mathcal{I}(g)(x)\, dx,
\end{align}
where  
$$\mathcal{I}g(x)=\int_{S^{4n-1}} g(x. (z_1, 0).(z_2, 0)^{-1})\, d\sigma(z_1, z_2).$$
We claim that $\mathcal{I}(g)(x) \leq C \|g\|_{L^{p_2}(\Ha)}.$ This in turn will conclude that $\mathfrak{S}: L^1(\Ha)\times L^{p_2}(\Ha)\to L^1(\Ha).$ We proceed to prove the claim. A further change of variables $z_{1}=\frac{\mathfrak{a}+\mathfrak{b}}{2}, z_{2}=\frac{\mathfrak{a}-\mathfrak{b}}{2}$ reduces the above to
\begin{equation*}
 \int_{S^{4n-1}} g(x. (\mf{b}, \frac{1}{4}\Im( \mf{b}. \mf{\bar{a}}))\, d\sigma(\mf{a}, \mf{b})  
\end{equation*}
which is same as $$\int_{\mathbb{R}^{2n}\times \mathbb{R}^{2n}} (\tau_{x}g)(\mf{b}, \frac{1}{4}\Im(\mf{b}. \mf{\bar{a}}))\, \delta(|\mf{a}|^2+|\mf{b}|^2-1) d\mf{a}\, d\mf{b}.$$  
Utilizing the the slicing argument \eqref{slicing-formula} we break the integral as
\begin{equation}
\label{endpt: slice}
 \int_{\mf{b}\in B^{2n}(0, 1)} (1-|\mf{b}|^2)^{n-1} \int_{S^{2n-1}} \tau_{x}g \bigg(\mf{b}, \frac{\sqrt{1-|\mf{b}|^2}}{4} \Im(\mf{b} . \mf{\bar{a}})\bigg) d\sigma_{2n-1}(\mf{a})\, d\mf{b}.    
\end{equation}
For a fixed $\mf{b}\in B^{2n}(0, 1),$ we call the inner integral $$\Upsilon(b):=\int_{S^{2n-1}} (\tau_{x}g)\bigg(\mf{b}, \frac{\sqrt{1-|\mf{b}|^2}}{4} \Im(\mf{b} . \mf{\bar{a}})\bigg)\, d\sigma_{2n-1}(\mf{a}).$$
For a non-zero $\mf{b},$ we write $\mf{b}=|\mf{b}| A\mf{e}_{1}$ for some orthogonal matrix $A\in O(2n)$ and $\mf{e}_{1}$ is the standard basis vector $(1, 0, \cdots, 0)\in \C^n.$ Therefore, using the rotational invariance of the surface measure, we obtain
\begin{align}
\nonumber\Upsilon(b)&:=\int_{S^{2n-1}} (\tau_{x}g)\bigg(\mf{b}, \frac{\sqrt{1-|\mf{b}|^2}}{4} |\mf{b}| (A\mf{e}_{1} . J\mf{\bar{a}})\bigg)\, d\sigma_{2n-1}(\mf{a})\\
\nonumber&\overset{\mf{a}\to J^{-1}\mf{a}}{=}\int_{S^{2n-1}} (\tau_{x}g)\bigg(\mf{b}, \frac{\sqrt{1-|\mf{b}|^2}}{4} |\mf{b}| (A\mf{e}_{1} . \mf{\bar{a}})\bigg)\, d\sigma_{2n-1}(\mf{a})\\
\nonumber&=\int_{S^{2n-1}} (\tau_{x}g)\bigg(\mf{b}, \frac{\sqrt{1-|\mf{b}|^2}}{4} |\mf{b}| (\mf{e}_{1} . A^{-1}\mf{\bar{a}})\bigg)\, d\sigma_{2n-1}(\mf{a})\\
&\overset{\mf{a}\to A\mf{a}}{=}\int_{S^{2n-1}} (\tau_{x}g)\bigg(\mf{b}, \frac{\sqrt{1-|\mf{b}|^2}}{4} |\mf{b}| {a}_{1}\bigg)\, d\sigma_{2n-1}(\mf{a}),
\label{enpt:slice:1}
\end{align}
where $ J :={ \scriptsize \begin{pmatrix}
0 & I_{n} \\
-I_{n} & 0 
\end{pmatrix} }$ is the matrix associated to the standard symplectic form on $\mathbb{R}^{2n}$ given by $\mf{b}^{T} J \mf{a}:= -\Im(\mf{b}. \mf{\bar{a}})$, $\mf{a}, \mf{b} \in \mathbb{C}^{n} \equiv \mathbb{R}^{2n}$.
Using spherical coordinates, we further decompose \eqref{enpt:slice:1} as
\begin{align}
\nonumber\Upsilon(b)=\nonumber&\int_{-1}^{1}\int_{S^{2n-2}} (\tau_{x}g)\bigg(\mf{b}, \frac{\sqrt{1-|\mf{b}|^2}}{4} |\mf{b}| s\bigg)\, (\sqrt{1-s^2})^{2n-3} d\theta\, ds\\
& = c_{n} \int_{-1}^{1}(\tau_{x}g)\bigg(\mf{b}, \frac{\sqrt{1-|\mf{b}|^2}}{4} |\mf{b}| s\bigg)\, (\sqrt{1-s^2})^{2n-3}\, ds.\label{enpt:slice:2}
\end{align}
Combining \eqref{endpt: slice} and \eqref{enpt:slice:2} we derive that
\begin{align*}
 \mathcal{I}g(x)=c_{n} \int_{\mf{b}\in B^{2n}(0, 1)}(1-|\mf{b}|^2)^{(n-1)} \int_{-1}^{1}(\tau_{x}g)\bigg(\mf{b}, {\textstyle \frac{\sqrt{1-|\mf{b}|^2}}{4} |\mf{b}| } s\bigg)\, (\sqrt{1-s^2})^{2n-3}\, ds \, d \mf{b}.
\end{align*}
Again, performing the change of variable $s\to \frac{4s}{|\mf{b}| \sqrt{1-|\mf{b}|^2} }$ the above expression  becomes a constant times
\begin{align}
 \int_{  \mf{b}\in B^{2n}(0, 1) } {  (  1-|\mf{b}|^2)^{(n-1)}  } \int_{|s|< \frac{1}{4}|\mf{b}|\sqrt{1-|\mf{b}|^2 }}(\tau_{x}g)(\mf{b}, s) {\big( { \textstyle 1-\frac{16s^2}{|\mf{b}|^2(1-|\mf{b}|^2)}   } \big) }^{\frac{(2n-3)}{2}}  { \textstyle   \frac{ds\, d\mf{b}}{|\mf{b}|\sqrt{1-|\mf{b}|^2}}  }.
\end{align}
Using H\"older's inequality this is dominated by
\begin{align*}
\|\tau_{x}g\|_{L^{p_2}(\Ha)} \left(\int_{|\mf{b}|<1} \int_{|s|< \frac{1}{4} |\mf{b}|\sqrt{1-|\mf{b}|^2 }} \frac{1}{|\mf{b}|^{p_{2}'}} d\mf{b}\, ds \right)^{1/p'_2}\leq C \|\tau_{x}g\|_{L^{p_2}(\Ha)}=C\|g\|_{L^{p_2}(\Ha)},  
\end{align*}
provided $-p_{2}'+2n+1> 0\iff p_{2}'< 2n+1$ which is equivalent to $p_{2} > \frac{2n+1}{2n}=\frac{d}{d-1}.$ This proves the claim, and thus, in turn completes the proof.  
\end{proof}

Our next result, for $p_{2}>\frac{d}{d-1},$ lifts the boundedness from $$\mathfrak{S}: L^{1}(\Ha)\times L^{p_2}(\Ha)\to L^{1}(\Ha)$$ to the H\"older bound
$$\mathfrak{S}: L^{1}(\Ha)\times L^{p_2}(\Ha)\to L^{s}(\Ha)$$ for all $s\in [\frac{p_2}{p_2+1}, 1].$ The proof essentially relies on the local nature of $\mathfrak{S}.$  The idea behind the proof goes back to work of Kenig--Stein \cite{KS}, see also \cite{IoseJGA}. This result will play an important role in propagating our $L^2(\Ha)\times L^2(\Ha)\to L^{1}(\Ha)$ decay estimate Proposition~\ref{lac piece 22 to 1} to the region $\mathcal{R}.$

\begin{lemma}
\label{KS}
Let $n\geq 2$. Then for each $\frac{d}{d-1} < p_{2} < \infty$  the operator $\mathfrak{S}$ maps $L^{1}(\Ha)\times L^{p_{2}}(\Ha)$ to $L^{s}(\Ha)$ boundedly, for all $s\in [\frac{p_2}{p_2+1}, 1].$ Similarly, for each $\frac{d}{d-1} < p_{1} < \infty$ the operator $\mathfrak{S}$ maps $L^{p_{1}}(\Ha)\times L^{1}(\Ha) $ to $L^{s}(\Ha)$ boundedly, for all $s\in [\frac{p_1}{p_1+1}, 1].$
\end{lemma}
\begin{proof} Fix $p_{2}$ such that $\frac{d}{d-1} < p_{2} < \infty$. Pick $\varepsilon >0$ so that $ \textstyle \frac{1}{p_{2}} = \frac{1}{d/(d-1) + \varepsilon} $. Set $ \textstyle \frac{1}{s_{0}} := 1 + \frac{1}{d/(d-1) + \varepsilon} $. Then from Proposition \ref{single:improv}, the case $s=1$, and the continuity of exponents in the complex interpolation it thus suffices to show
 $$\mathfrak{S}: L^{1}(\Ha)\times L^{p_{2}}(\Ha)\to L^{s_0}(\Ha).$$ For $a \in \Z^{2n+1}$, let $Q_{a} = a \cdot Q_{0}$, where $Q_{0}= [ 0,1)^{2n+1}$. The proof relies on two localization principles as follows:
\begin{itemize}
    \item[(P1)] There exists a dimensional constant $\kappa_{1}\geq 1$ such that $$\mathfrak{S}(f \chi_{Q_{a}} , g\chi_{Q_{b}})(x)=0\,\text{whenever}\, |b^{-1}a|\geq \kappa_1.$$
    \item[(P2)] There exists a dimensional constant $\kappa_{2}\geq 1$ such that $\supp(\mathfrak{S}(f, g))\subset (\supp(f)\cup \supp(g))\cdot B(0, \kappa_{2}).$
\end{itemize}
The properties $(P1), \text{and} (P2)$ are easy to verify and we leave the details. As $0<s_{0}< 1,$ $\|\mathfrak{S}(f, g)\|_{L^{s_0}(\Ha)}^{s_0}$ is dominated by
\begin{align}
\label{P1}
\int_{\Ha} \sum_{a \in \Z^{2n+1}}\sum_{b\in \Z^{2n+1}: |b|\leq \kappa_{1}}|\mathfrak{S}(f \chi_{Q_{a}}, g\chi_{Q_{a\cdot b}})(x)|^{s_0}\, dx.
\end{align}
By $(P2),$ for fixed $a, b\in \Z^{2n+1,}$ the support of $\mathfrak{S}(f \chi_{Q_{a}}, g\chi_{Q_{a\cdot b}})$ is contained in a finite measure set whose measure depends only on the dimension. Therefore, by H\"older's inequality, we lift the $L^{s_0}$ norm to $L^1$ and thus \eqref{P1} is further dominated by
\begin{align}
\nonumber \sum_{b\in \Z^{2n+1}: |b|\leq \kappa_{1}}\sum_{a \in \Z^{2n+1}}\left(\int_{\Ha} |\mathfrak{S}(f \chi_{Q_{a}}, g\chi_{Q_{a \cdot b}})(x)|\, dx\right)^{s_0}.  
\end{align}
At this point, we invoke $\mathfrak{S}: L^{1}(\Ha)\times L^{\frac{d}{d-1} + \varepsilon}(\Ha)\to L^{1}(\Ha)$ boundedness from Proposition~\ref{single:improv} to control the above by
\begin{align}
\nonumber \sum_{b\in \Z^{2n+1}: |b|\leq k_{1}}\sum_{a \in \Z^{2n+1}}\|f \chi_{Q_{a}}\|_{L^{1}(\Ha)}^{s_0} \|g\chi_{Q_{a\cdot b}})\|_{L^{\frac{d}{d-1} + \varepsilon}(\Ha)}^{s_0}.  
\end{align}
A further application of the H\"older's inequality and the bounded overlap of the family of cells $\{Q_{a}\}_{a\in \Z^{2n+1}},$ conclude that the above is bounded by $\|f\|_{L^{1}(\Ha)}^{s_0} \|g\|_{L^{\frac{d}{d-1} + \varepsilon}(\Ha)}^{s_0},$ and this completes the proof.
\end{proof}

\begin{proof}[Proof of Theorem \ref{single-scale-thm}]
Considering the scaling condition 
\begin{equation*}
    \mathfrak{S}_{r}(f,g)= \mathfrak{S}_{1} (f \circ \delta_{r}, g \circ \delta_{r}) \circ \delta_{1/r},
\end{equation*}
we will prove our results for $\mathfrak{S}$ only. For $1\leq p_1, p_2, p\leq \infty,$ that is for Banach range of exponents, we use  Minkowski's inequality to obtain $$\|\mathfrak{S}(f, g)\|_{L^p(\Ha)}\lesssim \|f\|_{L^{p_1}(\Ha)}\|g\|_{L^{p_2}(\Ha)},$$
for $\frac{1}{p}=\frac{1}{p_1}+\frac{1}{p_2}.$ Interpolating these with the estimates 
$\mathfrak{S}: L^{1}(\Ha)\times L^{p_2}(\Ha)\to L^{s}(\Ha),$ for $s\in [\frac{p_2}{p_2+1}, 1],$
and $\mathfrak{S}: L^{p_1}(\Ha)\times L^{1}(\Ha)\to L^{s}(\Ha),$ for $s\in [\frac{p_1}{p_1+1}, 1]$ with any $p_{1},p_2>\frac{d}{d-1},$ from Lemma~\ref{KS} we conclude that $\mathfrak{S}$ is bounded from $L^{p_1}(\Ha)\times L^{p_2}(\Ha)$ to $L^{p}(\Ha),$ for $(\frac{1}{p_1}, \frac{1}{p_2})$ in the open pentagon $\mathcal{D}$ including the boundary line segments $[O,A], [A,E),$  $[D,F)$ and $[D,O]$ and $\frac{1}{p}=\frac{1}{p_1}+\frac{1}{p_2}$.

\end{proof}

\section{Full maximal function}
\label{sec:full}

This section contains the proof of the sharp Lebesgue space estimates for the bilinear Nevo--Thangavelu maximal function, Theorem~\ref{full}. The reader will notice, while decoupling the bilinear averages using the slicing techniques, the Birkhoff averages over linear spherical means appear naturally. 

\begin{proof}[Proof of Theorem~\ref{full}] 
We start with the demonstrating the slicing argument in this case. Let $r>0$ be any fixed positive number. Then using \eqref{slicing-formula} we obtain the following:
\begin{equation}\label{slicing formula}
    \begin{split}
     &\mathfrak{S}_r(f,g)(x)\\
     &=\int_{B^{2n}(0,1)}f(x. \delta_{r}(z_1, 0)^{-1})\int_{S^{2n-1}}  g(x. \delta_{r}(( { \textstyle \sqrt{1-\|z_1\|^2} z_{2} } , 0)^{-1})\, d\sigma_{2n-1}(z_2) { \textstyle (1-\|z_1\|^2)^{n-1} dz_{1} }  \\
     &=\int_{0}^{1}\int_{S^{2n-1}} f(x. \delta_{r}(s\omega, 0))  d\sigma_{2n-1}(\omega)\\
     &  \ \ \ \ \ \ \ \ \ \ \ \  \times   \int_{S^{2n-1}} g(x. \delta_{r}(\sqrt{1-s^2} z_{2}, 0)^{-1})d\sigma_{2n-1}(z_2) (1-s^2)^{(n-1)} s^{2n-1} \, ds\\
     &=\int_{0}^{1} s^{2n-1}(1-s^2)^{n-1} A_{rs}f(x) A_{r\sqrt{1-s^2}}g(x)\, ds\\
     &\leq M_{S}(f)(x)\, \int_{0}^{1} s^{2n-1}(1-s^2)^{n-1} A_{r\sqrt{1-s^2}}g(x)\, ds\\
     &= M_{S}(f)(x)\, \int_{0}^{1} s^{2n-2}(1-s^2)^{n-1} A_{r\sqrt{1-s^2}}g(x)\, sds.
    \end{split}
\end{equation}
Performing a change of variable $s\to r^{-1}\sqrt{(r^2-s^2)}$ we obtain that $\mathfrak{S}_{r}(f, g)(x)$ is dominated by 
$$M_{S}(f)(x) \frac{1}{r^2}\int_{0}^{r}   {  \Big(\frac{r^2-t^2}{r^2}\Big)^{n-1} \Big(\frac{t^2}{r^2}\Big)^{n-1}   }  A_{t}(g)(x)\, t\, dt\lesssim M_{S}(f)(x) \frac{1}{r}\int_{0}^{r} A_{t}g(x)\, dt.$$

Now taking supremum over $r>0,$ we obtain that $$\mathfrak{M}_{\text{full}}(f, g)(x)\leq C M_{S}(f)(x)\Lambda(g)(x),$$ 
where $M_{S}$ is the Nevo--Thangavelu maximal operator and $\Lambda$ is the maximal operator over uniform averages in \eqref{Birkhoff avg}. Interchanging the role of $f$ and $g,$ we conclude 
\begin{align}
\label{Slicing-full}
\mathfrak{M}_{\text{full}}(f, g)(x)\leq C \min\{M_{S}(f)(x)\Lambda(g)(x), M_{S}(g)(x)\Lambda(f)(x)\}.    
\end{align}

Thus, if $\frac{d-1}{d-2} < p_{1} \leq \infty, 1< p_{2} \leq \infty$ then from H\"older's inequality we obtain
\begin{equation}\label{full: max 1}
   \| \mathfrak{M}_{\text{full}}(f, g) \|_{L^{p}(\Ha)} \lesssim  \| M_{S}f \|_{L^{p_{1}}(\Ha)} \, \| \Lambda g \|_{L^{p_{2}}(\Ha)}  \lesssim \| f \|_{L^{p_{1}}(\Ha)} \, \| g \|_{L^{p_{2}}(\Ha)}.
\end{equation}
Similarly, for $1< p_{1} \leq \infty, \frac{d-1}{d-2} < p_{2} \leq \infty,$ we can use the estimate $\mathfrak{M}_{\text{full}}(f, g)(x)\leq C \Lambda(f)(x) M_{S}g(x)$ and H\"older's inequality to obtain 
\begin{equation}\label{full: max 2}
   \| \mathfrak{M}_{\text{full}}(f, g) \|_{L^{p}(\Ha)} \lesssim \| \Lambda f \|_{L^{p_{1}}(\Ha)} \| M_{S}g \|_{L^{p_{2}}(\Ha)} \,   \lesssim \| f \|_{L^{p_{1}}(\Ha)} \, \| g \|_{L^{p_{2}}(\Ha)}.
\end{equation}
Now a standard linearization technique and complex interpolation between \eqref{full: max 1} and \eqref{full: max 2} we obtain that (\ref{full: max 1}) holds when $(\frac{1}{p_{1}}, \frac{1}{p_{2}})$ belong to the open pentagon with the corners $(0,0)$, $(0,1)$, $(\frac{d-2}{d-1},1)$, $(1,\frac{d-2}{d-1})$ and $(1,0)$ with half open line segments $[O,A)$ and $[O,D)$ included.

\end{proof}

\section{Lacunary maximal function}
\label{sec:lac}
This section contains all the details of the boundedness of the lacunary maximal function $\mf{M}_{\text{lac}}.$ As mentioned earlier, a key ingredient in the proof is the $L^2(\Ha)\times L^2(\Ha)\to L^1(\Ha)$ decay estimate for ``localized" pieces of the averaging operator $\mf{M}.$ Before proceeding further, we borrow some terminology from \cite{HickmanJFA}. A locally finite measure $\nu$ on $\Ha$ is said to satisfy the \textit{Curvature assumption} (CA) if there exists  natural number $N$ such that the iterated convolution $\nu^{(N)}$ is absolutely continuous with respect to the Haar measure of $\Ha$ and there exists $\gamma>0, C_{\mu}>0$ such that the Radon-Nikodym derivative of $\nu^{(N)}$, say $h,$ satisfies 
\begin{align}
\label{CA}
\int_{\Ha}\big(|h(xy^{-1})-h(x)|+|h(y^{-1} x)-h(x)|\big)\, dx\leq C_{\mu}\, |y|^{\gamma}   
\end{align}
for all $y\in \Ha.$ The iterated convolutions are defined as $\nu^{(2k+1)}=\nu^{(2k)}\ast \tilde{\nu}$ and $\nu^{(2k)}=\nu^{(2k-1)}\ast \nu$ for $k\in \mathbb{N}.$ Moreover, from Lemma~3.3 in \cite{HickmanJFA} we know that the surface measure of the horizontal sphere $S^{2n-1}\times \{0\},$ $\mu=\sigma_{2n-1}\ast \delta_{0}$ in $\Ha$, satisfies \eqref{CA}. We shall frequently make use of the following scaling relations:
\begin{equation*}
\mf{S}_{r}(f,g)= \mf{S} (f \circ \delta_{r}, g \circ \delta_{r}) \circ \delta_{1/r},  \quad  \text{and}\quad
    (f \ast \psi_{t}) \circ \delta_{r}= f \circ \delta_{r} \ast \psi_{t/r}.
\end{equation*}
By slicing argument we have already seen that 
$$\mf{S}(f, g)(x)= \int_{0}^{1}  s^{2n-1}(1-s^2)^{n-1} A_{s}f(x)\, A_{\sqrt{1-s^2}}g(x)\, ds.$$
The next step is to introduce a Littlewood-Paley decomposition (see Proposition~\ref{LP}) of $f$ and $g$ to write $f=\sum_{j} f\ast \psi_{2^j}$ and $g=\sum_{k} g\ast \psi_{2^k}$. Consequently, 
$$\mf{S}(f, g)(x)=\sum\limits_{k, j} \mf{S}^{j, k}(f, g)(x),$$ where
\begin{equation*}
    \mf{S}^{j, k}(f, g)(x)= \int_{0}^{1}  s^{2n-1}(1-s^2)^{n-1} A_{s} (f \ast \psi_{2^{j}})(x)\, A_{\sqrt{1-s^2}} ( f\ast \psi_{2^{k}})(x)\, ds. 
\end{equation*}

Now we state our main decay estimate.
\begin{prop}\label{lac piece 22 to 1} There exists $\delta>0$ so that
 \begin{equation}
    \|\mf{S}^{j,k}(f,g)\|_{L^1(\Ha)}\lesssim 2^{(j+k)\delta}\|f\|_{L^2(\Ha)}\|g\|_{L^2(\Ha)},
\end{equation}   
for all integers $j,k$.
\end{prop}

\begin{remark}
The reader should note that the above estimate is only useful for integers $j, k<0.$ This should not be surprising as the negative indices correspond to the ``high frequencies" in an analogy with the Euclidean case.       
\end{remark}

\begin{proof}
Using Minkowski and H\"older inequality, we obtain   
\begin{equation}
    \begin{split}
& \|\mf{S}^{j,k}(f, g)\|_{L^1(\Ha)} \\
&\leq \int_{0}^{1}  s^{2n-1}(1-s^2)^{n-1} \| A_{s} (f \ast \psi_{2^{j}})\, A_{\sqrt{1-s^2}} ( g\ast \psi_{2^{k}})   \|_{L^1(\Ha)}\, ds\\
&\leq \int_{0}^{1}  s^{2n-1}(1-s^2)^{n-1} \|   A_{s} (f \ast \psi_{2^{j}}) \|_{L^2(\Ha)}\, \| A_{\sqrt{1-s^2}} ( g\ast \psi_{2^{k}})    \|_{L^2(\Ha)}\, ds.        
    \end{split}
\end{equation}
Suppose for the moment that we could prove the following claim:\\
\noindent\textbf{Claim:}  There exists some $0< \delta < 2n$ such that 
\begin{equation} \label{jth piece f}
    \|A_{s} (f \ast \psi_{2^{j}})\|_{L^2(\Ha)}\leq C (2^{j}/s)^{\delta} \| f \|_{L^2(\Ha)}.
\end{equation}
Then we would have
\begin{equation*}
   \begin{split}
  &\|\mf{S}(f, g)\|_{L^1(\Ha)}\\
 & \lesssim \|f\|_{L^2(\Ha)}\|g\|_{L^2(\Ha)}  \,  \int_{0}^{1}  s^{2n-1}(1-s^2)^{n-1} 2^{j\delta} s^{-\gamma} 2^{k\delta} (1-s^2)^{-\gamma /2} ds\\
 &\lesssim \|f\|_{L^2(\Ha)}\|g\|_{L^2(\Ha)} \,  2^{\delta(j+k)},   
   \end{split} 
\end{equation*}
for $\delta<2n$.

It thus remains to prove the (\ref{jth piece f}) in the claim. As the Fourier transform on the Heisenberg group is operator valued, we shall pursue the approach based on the iterated $T^{\ast} T$ method.
By scaling, $A_{s} (f \ast \psi_{2^{j}}) = A (f \circ \delta_{s} \, \ast \, \psi_{2^{j}/s}) \circ \delta_{1/s} $, the estimate  (\ref{jth piece f}) is equivalent to
\begin{equation} \label{jth piece f II}
    \|A (f \ast \psi_{2^{j}/s})\|_{L^2(\Ha)}\leq  C ( 2^{j}/s)^{ \delta} \| f \|_{L^2(\Ha)},
\end{equation}
where $A$ denotes $A_{1}.$ So by $T^{\ast} T$ method it suffices to show that
 \begin{equation} \label{jth piece f III}
     \|f \ast \psi_{2^{j}/s} \ast \mu  \ast \tilde{ \mu} \ast \tilde{\psi}_{2^{j}/s}  \|_{L^2(\Ha)}\leq  (C \, 2^{j}/s)^{2 \delta} \| f \|_{L^2(\Ha)}.
 \end{equation}
To this end, let $N \geq 1$ be such that the convolution of $\mu$ iterated $N$ times, denoted as $\mu^{(N-1)}$, is absolutely continuous with respect to the Haar measure on $\mathbb{H}^n.$ We define the operator $A^{(N)}f :=f\ast \psi_{2^{j}/s}\ast \mu^{(N)}$, and observe that
\begin{equation*}
  (A^{(N)})^{\ast}  A^{(N)} f = f \ast \psi_{2^{j}/s} \ast \mu^{(N+1)}  \ast \widetilde{ \mu^{(N-1)} } \ast \tilde{\psi}_{2^{j}/s}.
\end{equation*}
Young’s inequality now concludes
\begin{equation*}
 \begin{split}
\scriptsize \| (A^{(N)})^{\ast}  A^{(N)} f\|_{L^2(\Ha)}
& \scriptsize \leq \|   A^{(N+1)} \|_{{L^2(\Ha)}\to {L^2(\Ha)}} \,  \| \psi_{2^{j}/s} \ast \mu^{(N-1)} \|_{L^1(\Ha)}  \, \| f \|_{L^2(\Ha)},
 \end{split}
\end{equation*}
whose recursive application yields the following  operator norm estimate
\begin{equation*}
    \| A^{(1)}  \|_{{L^2(\Ha)}\to {L^2(\Ha)}} \lesssim_{N} \| \psi_{2^{j}/s} \ast \mu^{(N-1)} \|_{L^1(\Ha)}^{1/ 2^{N}}.
\end{equation*}
Whereas, the $L^1$ norm $\| \psi_{2^{j}/s} \ast \mu^{(N-1)} \|_{1}$ can be estimated using the cancellation present in $\psi$ and by the \textit{curvature condition \eqref{CA}} on $\mu,$ as 
\begin{equation*}
\begin{split}
\int_{\Ha} &  \left|\int_{\Ha}\left( \mu^{(N-1)}(y^{-1} \cdot x) - \mu^{(N-1)}(x)\right) \psi_{2^{j}/s}(y)\, dy \right|\,\ud x \\
& \lesssim \int_{\Ha}|y|^{\gamma} |\psi_{2^{j}/s}|(y)\lesssim  (2^{j}/s)^{\gamma},     
\end{split}    
\end{equation*}
which gives $ \| A^{(1)} f \|_{L^2(\Ha)} \lesssim_{N} (2^{j}/s)^{\gamma / N} \| f \|_{L^2(\Ha)}$, as it was required in (\ref{jth piece f III}). 
\end{proof}

\subsection{Single scale decay estimates throughout $\mathcal{D}$}
 Next, we interpolate this decay estimate (when $j,k<0$) with the following single-scale estimates coming from Theorem \ref{single-scale-thm}. As the functions $\psi_{2^{\ell}}$ are $L^1$ normalized, Theorem \ref{single-scale-thm} followed by the Young's inequality implies that 
 \begin{equation}
 \label{Propagation:1}
    \|\mathfrak{S}^{j,k}(f,g)\|_{p} \lesssim  \|f\ast \psi_{2^j}\|_{ p_{1} } \|g\ast \psi_{2^k}\|_{ p_{2}}\lesssim \|f\|_{ p_{1} } \|g\|_{ p_{2}},
\end{equation}
holds uniformly in $j,k$, for all $(\frac{1}{p_1}, \frac{1}{p_2})\in \mathcal{D}.$ Interpolation of \eqref{Propagation:1} with Proposition~\ref{lac piece 22 to 1} yields the following proposition.

\begin{prop}\label{lac piece p1 p2 to p in R} Let $n \geq 2$. Suppose $\big(\frac{1}{p_1}, \frac{1}{p_2}\big)$ is contained in the interior of the region $\mathcal{D}$ of  Theorem~\ref{single-scale-thm}. Let $0<p < \infty$ be in the H\"older relation with $p_{1},p_{2}$, that is, $\frac{1}{p} = \frac{1}{p_1} + \frac{1}{p_2}$. Then, there exists $\delta:=\delta(p_1, p_2, n)>0$ so that
 \begin{equation}
    \|\mathfrak{S}^{j,k}(f,g)\|_{p} \lesssim 2^{(j+k)\delta}\|f\|_{p_{1}} \|g\|_{p_{2}},
\end{equation}   
for all integers $j,k$.
\end{prop}

\subsection{Proof of Theorem \ref{Horizontal lac result}}
\medskip

We now have enough ingredients to prove Theorem \ref{Horizontal lac result}. Consider first proving the boundedness of $\mathfrak{M}_{\mathrm{lac}}$ when the points $(\frac{1}{p_{1}}, \frac{1}{p_{2}})$ are in the interior of the diagram $\mathcal{R}$. We  will include the remaining boundary points of $\mathcal{R}$ later in the proof.

Let $1 \leq p < \infty$. Fix an $\ell \in \mathbb{Z}$. By non-negativity of $\mathfrak{M}_{\mathrm{lac}}$, we may assume $f,g$ are non-negative.  We begin as before with a decomposition of identity from Lemma \ref{LP}, which we can rewrite as
\begin{equation*}
 f=  f\ast \phi_{2^{\ell}}  + \sum_{j<0} f\ast \psi_{2^{j + \ell}}  \ \ \ \text{and} \ \ \   g=  g \ast \phi_{2^{\ell}}  + \sum_{k<0} g\ast \psi_{2^{k+l}}.
\end{equation*}
We thus decompose the single lacunary average, $\mathfrak{S}_{2^{l}}$, as
\begin{equation} \label{Hor l piece decom}
\begin{split}
\mathfrak{S}_{2^{\ell}}(f, g)(x)=\mathfrak{S}_{2^{\ell}}(f  \ast \phi_{2^{\ell}}, g) & + \mathfrak{S}_{2^{\ell}}(f, g\ast \phi_{2^{\ell}}) - \mathfrak{S}_{2^{\ell}}(f\ast \phi_{2^{\ell}}, g \ast \phi_{2^{\ell}}) \\
& + \sum_{ j, k<0} \mathfrak{S}_{2^{l}}(f \ast \psi_{2^{j+ \ell}}, g \ast \psi_{2^{k + \ell}}).  
\end{split} 
\end{equation}

We notice that the slicing analysis, as in (\ref{slicing formula}), reveals that we may dominate the first term by the product of Hardy-Littlewood maximal function $M_{\mathrm{HL}}(f)$, and the maximal function $\Lambda(g)$ associated to the uniform averages of the measure $\mu_{s}$. To see this, from scaling
\begin{equation*}
\mathfrak{S}_{2^{\ell}}(f\ast \phi_{2^{\ell}}, g) = \mathfrak{S} (f \circ \delta_{2^{\ell}} \, \ast \, \phi, g \circ \delta_{2^{\ell}}) \circ \delta_{1/2^{\ell}}    
\end{equation*}
it suffices to show
\begin{equation*}
|\mathfrak{S}( f \ast \phi, g)(x)| \lesssim M_{\mathrm{HL}}(f)(x) \ \Lambda(g)(x),   
\end{equation*}
because by rescaling it, the first term is then bounded by the product  $M_{\mathrm{HL}}(f \circ \delta_{2^{\ell}}) ( \delta_{2^{- \ell} } x)$ $   \Lambda (  g \circ \delta_{2^{\ell}}) ( \delta_{2^{- \ell} } x)$ which is further bounded by $M_{\mathrm{HL}}(f)(x) \ \Lambda(g)(x)$.

Then from  properties of $\phi$, and the slicing, the left hand side is
\begin{equation*}\label{slicing formula ergodic dom}
    \begin{split}
     |\mathfrak{S}( f \ast \phi, g)(x)|
     & \leq \int_{0}^{1} s^{2n-1}(1-s^2)^{n-1} |A_{\sqrt{1-s^2}}(f \ast \phi)(x)|\, A_{s}(g)(x)  ds\\
     & \lesssim \int_{0}^{1} s^{2n-1}(1-s^2)^{n-1} (f \ast  \chi_{B(0, C +  \sqrt{1-s^2})})(x)\, A_{s}(g)(x)  ds\\
     &\leq M_{\mathrm{HL}}(f)(x)\, \int_{0}^{1} s^{2n-1}(1-s^2)^{n-1} A_{s}(g)(x) \, ds\\
     &\leq M_{\mathrm{HL}}(f)(x)\,   \int_{0}^{1}  (g \ast \mu_{s})(x)\, ds  \\
     &\leq M_{\mathrm{HL}}(f)(x)\, \sup_{r>0} \frac{1}{r} \int_{0}^{r}  (g \ast \mu_{s})(x)\, ds  \\
     & = M_{\mathrm{HL}}(f)(x)\, \Lambda(g)(x).
    \end{split}
\end{equation*}
Performing a similar set of computations, we may control the second term in (\ref{Hor l piece decom}) by $\Lambda(f)(x) \, M_{\text{HL}}g(x).$ Finally, the third term there can also be controlled by
$ M_{\mathrm{HL}}(f)(x) \, M_{\mathrm{HL}}(f)(x)$. 

Altogether, from H\"older's inequality and the $L^{q}(\Ha)$-boundedness of Hardy-Littlewood maximal function and that of $\Lambda$, see Proposition \ref{Birkhoff avg},  for all $1<q \leq \infty$, we have
\begin{equation}
\begin{split}
 \| \sup_{\ell} \left\lvert \mathfrak{S}_{2^{\ell}}(f  \ast \phi_{2^{\ell}}, g)  + \mathfrak{S}_{2^{\ell}}(f, g\ast \phi_{2^{\ell}}) - \mathfrak{S}_{2^{\ell}}(f\ast \phi_{2^{\ell}}, g \ast \phi_{2^{\ell}}) \right\rvert \|_{p} \lesssim \| f \|_{p_{1}} \, \| g \|_{p_2}.   
\end{split}    
\end{equation}

Next, the treatment of the high frequencies in (\ref{Hor l piece decom}), can be done by an argument essentially due to Michael Christ (see for instance, \cite{BF24}). For each $L\in \mathbb{N},$ let $\Upsilon(L)$ be the smallest constant such that
\begin{equation*}
{\textstyle  \left\| \sup_{|\ell| \leq L}  \left\lvert   \mathfrak{S}_{2^{l}}(f , g ) \right\rvert  \right\|_{p}  \leq \Upsilon(L)  \| f \|_{p_{1}} \, \| g \|_{p_2} },
\end{equation*}
then to show the boundedness of $\mathfrak{M}_{\text{lac}}$ it then suffices to show that $\Upsilon(L) \lesssim 1$. We have the following two vector valued inequalities: On one hand, from the positivity of $\mathfrak{S}_{2^{\ell}}$, majorization by Hardy-Littlewood maximal function gives
\begin{equation}\label{vect ineq I}
{\textstyle  \left\| \sup_{|\ell| \leq L}  \left\lvert    \mathfrak{S}_{2^{l}}(f_{\ell}, g_{\ell} ) \right\rvert  \right\|_{p}  \leq \Upsilon(L)  \| \sup_{|\ell| \leq L} |f_{\ell}| \|_{p_{1}} \, \| \sup_{|\ell| \leq L} |g_{\ell}| \|_{p_{2}} }.
\end{equation}
On the other hand the scale invariance of $\mathfrak{S}_{2^{\ell}}$ in conjugation with decay estimate of $\mathfrak{S}^{j,k}$ from  Proposition \ref{lac piece p1 p2 to p in R}, give
\begin{align} \label{vect ineq II}
\nonumber&\textstyle \left\| \left( \sum_{|\ell| \leq L}  \left\lvert   \mathfrak{S}_{2^{l}}(f_{\ell} \ast \psi_{2^{j + \ell }}, g_{\ell} \ast \psi_{2^{k + \ell }})  \right\rvert^{p} \right)^{1/p} \right\|_{p}\\
&\leq 2^{(j+k) \delta} \| \sum_{|\ell| \leq L} |f_{\ell}| \|_{p_{1}} \, \| \sum_{|\ell| \leq L} |g_{\ell}| \|_{p_{2}} .
\end{align}
Interpolation of (\ref{vect ineq I}) (with $f_{\ell}$ replaced by $ f_{\ell} \ast \psi_{2^{j+ \ell}}$), and (\ref{vect ineq II}), yields a key estimate
\begin{equation} \label{square fn typ}
\begin{split}
 & \textstyle \left\| \left( \sum_{|\ell| \leq L}  \left\lvert   \mathfrak{S}_{2^{l}}   (f_{\ell} \ast \psi_{2^{j + \ell}}, g_{\ell} \ast \psi_{2^{k + \ell}})  \right\rvert^{2p} \right)^{1/2p} \right\|_{p}  \\
 & \textstyle \leq 2^{(j+k) \delta /2} \Upsilon(L)^{1/2}  \| \left( \sum_{|\ell| \leq L} |f_{\ell}|^2 \right)^{1/2} \|_{p_{1}} \, \| \left( \sum_{|\ell| \leq L} |g_{\ell}|^2 \right)^{1/2} \|_{p_{2}} .    
\end{split}
\end{equation}

\noindent By applying (\ref{square fn typ}) to $f_{\ell}:= f \ast \psi_{2^{j^{\prime} + \ell}}$, in conjugation with union bound argument and Littlewood-Paley theory, we may bound
\begin{equation} \label{term to decay}
\begin{split}
\textstyle   \left\| \sup_{|\ell| \leq L}  \left\lvert   \mathfrak{S}_{2^{l}}(f \ast \psi_{2^{j^{\prime} + \ell}} \ast \psi_{2^{j + \ell}}, g_{\ell} \ast \psi_{2^{k + \ell}}) \right\rvert  \right\|_{p}
\end{split}    
\end{equation}
by  a positive constant times
\begin{equation}\label{decay in j}
\begin{split}
& \textstyle  2^{j \delta /2}  2^{k \delta /2} \Upsilon(L)^{1/2} \| \left( \sum_{|\ell| \leq L} |f \ast \psi_{2^{j^{\prime} + \ell}}|^2 \right)^{1/2} \|_{p_{1}} \, \| \left( \sum_{|\ell| \leq L} |g_{\ell}|^2 \right)^{1/2} \|_{p_{2}} \\
& \textstyle \lesssim 2^{j \delta /2}  2^{k \delta /2} \Upsilon(L)^{1/2} \| f \|_{p_{1}} \, \| \left( \sum_{|\ell| \leq L} |g_{\ell}|^2 \right)^{1/2} \|_{p_{2}}. 
\end{split}    
\end{equation}
Interchanging the role of $j$ and $j^{\prime}$ (that is assume (\ref{square fn typ}) holds for $j^{\prime}$ instead of $j$ and then take $f_{\ell}:= f \ast \psi_{2^{j + \ell}}$), we may also bound (\ref{term to decay}) by
\begin{equation*}\label{decay in j prime}
\begin{split}
 \textstyle  2^{j^{\prime} \delta /2}  2^{k \delta /2} \Upsilon(L)^{1/2} \| f \|_{p_{1}} \, \| \left( \sum_{|\ell| \leq L} |g_{\ell}|^2 \right)^{1/2} \|_{p_{2}}. 
\end{split}    
\end{equation*}
Taking the geometric mean of this with (\ref{decay  in j}), we obtain that (\ref{term to decay}) is bounded by
\begin{equation*}\label{decay j and j prime}
\begin{split}
 \textstyle    2^{(j+ j^{\prime}) \delta /2}  2^{k \delta /2} \Upsilon(L)^{1/2} \| f \|_{p_{1}} \, \| \left( \sum_{|\ell| \leq L} |g_{\ell}|^2 \right)^{1/2} \|_{p_{2}}.
\end{split}    
\end{equation*}
If we then sum this in $j<0$ and $j^{\prime}<0$, we obtain
\begin{equation} \label{term to decay in k}
\begin{split}
& \textstyle   \left\| \sup_{|\ell| \leq L}  \left\lvert \sum_{j, j^{\prime}<0} \,    \mathfrak{S}_{2^{l}}(f \ast \psi_{2^{j^{\prime} + \ell}} \ast \psi_{2^{j + \ell}}, g_{\ell} \ast \psi_{2^{k + \ell}}) \right\rvert  \right\|_{p} \\
& \textstyle \lesssim 2^{k \delta /2} \Upsilon(L)^{1/2} \| f \|_{p_{1}} \, \| \left( \sum_{|\ell| \leq L} |g_{\ell}|^2 \right)^{1/2} \|_{p_{2}}. 
\end{split}    
\end{equation}

\noindent A similar argument employed to (\ref{term to decay in k}), with $g_{\ell}:= g \ast \psi_{2^{k^{\prime} + \ell}}$ yield
\begin{equation} \label{term nothing to decay}
\begin{split}
& \textstyle   \left\| \sup_{|\ell| \leq L}  \left\lvert \sum_{j, j^{\prime} , k, k^{\prime}  <0} \,    \mathfrak{S}_{2^{l}}(f \ast \psi_{2^{j^{\prime} + \ell}} \ast \psi_{2^{j + \ell}}, g \ast \psi_{2^{k^{\prime} + \ell}} \ast \psi_{2^{k + \ell}}) \right\rvert  \right\|_{p} \\ &  \lesssim  \Upsilon(L)^{1/2} \| f \|_{p_{1}} \, \| g \|_{p_{2}}.  
\end{split}    
\end{equation}
 
Since, analogous to (\ref{Hor l piece decom}), we may also write
\begin{equation} \label{Hor l piece double decom}
\begin{split}
\mathfrak{S}_{2^{\ell}}(f, g)(x) & =\mathfrak{S}_{2^{\ell}}(f  \ast (2 \phi - \phi \ast \phi)_{2^{\ell}}, g)  + \mathfrak{S}_{2^{\ell}}(f, g\ast (2 \phi - \phi \ast \phi)_{2^{\ell}}) \\
& - \mathfrak{S}_{2^{\ell}}(f\ast (2 \phi - \phi \ast \phi)_{2^{\ell}}, g \ast (2 \phi - \phi \ast \phi)_{2^{\ell}}) \\
& \textstyle + \sum_{j, j^{\prime} , k, k^{\prime}  <0} \,    \mathfrak{S}_{2^{l}}(f \ast \psi_{2^{j^{\prime} + \ell}} \ast \psi_{2^{j + \ell}}, g \ast \psi_{2^{k^{\prime} + \ell}} \ast \psi_{2^{k + \ell}}).  
\end{split} 
\end{equation}
We may as before control the terms involving $\phi$ by maximal functions $M_{\mathrm{HL}}$ and $\Lambda$, and the last term by (\ref{term nothing to decay}) is controlled by constant times $\Upsilon(L)^{1/2} \| f \|_{p_{1}} \, \| g \|_{p_{2}}$.

As a consequence, we finally obtain the estimate
\begin{equation*}
{\textstyle  \left\| \sup_{|\ell| \leq L}  \left\lvert   \mathfrak{S}_{2^{l}}(f , g ) \right\rvert  \right\|_{p}  \leq \Upsilon(L)^{1/2}  \| f \|_{p_{1}} \, \| g \|_{p_2} },
\end{equation*}
which forces $ \Upsilon(L) \lesssim \Upsilon(L)^{1/2}$, thus concluding that
\begin{equation*}
  \Upsilon(L) \lesssim 1,
\end{equation*}
which is what we were aiming for. Finally, sub-additivity of $\| \cdot \|^{\min (1,p)}$ takes care of  the case $0<p<1$ as well. 

Now it only remains to consider the boundedness of $\mathfrak{M}_{\mathrm{lac}}$ for boundary exponents, $(\frac{1}{p_{1}}, \frac{1}{p_{2}})$, lying in either of the line segments $[O,A), $ and $[O, D)$. However, these are covered immediately due to the pointwise domination $\mathfrak{M}_{\mathrm{lac}}(f,g)(x) \leq \mathfrak{M}_{\mathrm{full}}(f,g)(x)$ and the bounds for $\mathfrak{M}_{\mathrm{full}}$ for these exponents described in Theorem \ref{full}. This completes the proof of the Theorem \ref{Horizontal lac result}.

\section{Sharpness}
\label{sec:sharp}
In this section, we show that the Theorem~\ref{full} is sharp. We construct Knapp type examples to conclude the sharpness. Our examples are motivated by \cite{LeeJFA} and \cite{Roos}. Our second result provides another necessary condition for the boundedness of the local maximal function $\mf{M}_{\text{loc}}.$

\begin{proof}[Proof of Proposition~\ref{nece:loc}]

Let $\delta>0$ be any small positive number and  we set
\[
f= \chi_{B^{e}(0,\delta)} \quad \text{and} \quad
g = \chi_{B^{e}(0,\delta)},
\]
where $B^{e}(0,\delta)$ represents the euclidean ball of radius $\delta$ centered at origin in $\mathbb{H}^n$. Let $x=(z, t)\in \mf{R}_{\delta},$ where \[
\mf{R}_{\delta}:= \left\{ (z, t) \in \Ha : \frac{1}{\sqrt{2}} \le \|z\| \le
\frac{1}{\sqrt{2}} + \kappa, |t|<\delta \right\},
\]
and $\kappa$ is sufficiently small but fixed positive number. We will show that for $x=(z, t)\in \mf{R}_{\delta},$
\begin{equation}
\label{3.8}
\mathfrak M_{\text{loc}}(f, g)(x)\ge \mathfrak{S}_{\sqrt{2}\|z\|}(f, g)(x)\geq C \delta^{4n-1}.
\end{equation}

Let $(z, t)\in \mathfrak{R}_{\delta},$ consider two sets
\[
E_{(z, t)}^{1} := \{z_{1}\in B^{2n}(0, 1) :
\|\hat{z} - \sqrt{2}z_{1} \|
\le c\delta \},\quad E_{(z, t)}^{2} :=
\{ z_{2} \in {S}^{2n-1}: \| \hat{z} - z_{2} \|
\le c' \delta\},
\]
here $\hat{z}$ represents the unit vector $\hat{z}:=z/\|z\|.$ By the slicing argument we obtain
\begin{align*}
&\mathfrak{S}_{\sqrt{2}\|z\|}(f,g)(x)\\
     &=\int_{B^{2n}(0,1)}f(x. \sqrt{2}\|z\|(z_1, 0)^{-1})\cdot \\
     &\int_{S^{2n-1}} g \left( x. \sqrt{2}\|z\|   \  (\sqrt{1-\|z_1\|^2} z_{2}, 0)^{-1}   \right)\,(1-\|z_1\|^2)^{n-1} d\sigma_{2n-1}(z_2) dz_{1}\\
     &\geq \int_{E^1_{(z, t)}}f(x. \sqrt{2}\|z\|(z_1, 0)^{-1})\cdot\\
     &\int_{E_{(z, t)}^{2}} g\big(x. \sqrt{2}\|z\|  \, (\sqrt{1-\|z_1\|^2} z_{2}, 0)^{-1}\big)\,(1-\|z_1\|^2)^{n-1} d\sigma_{2n-1}(z_2) dz_{1}.
\end{align*}

Observe a series of computations. Firstly, as $z_{1}\in E_{(z, t)}^{1}$, by triangle inequality, we obtain that $|\frac{1}{\sqrt{2}}-\|z_{1}\||\lesssim \|\hat{z}-\sqrt{2} z_1\|=O(\delta),$ and this in turn yields
\begin{align}
\nonumber&\big|\sqrt{1-\|z_1\|^2}-\frac{1}{\sqrt{2}}\big| \big|\sqrt{1-\|z_1\|^2}+\frac{1}{\sqrt{2}}\big|\\
&=\big(1-\|z_{1}\|^2-\frac{1}{2}\big)=\big(\frac{1}{2}-\|z_{1}\|^2\big)=\big|\frac{1}{\sqrt{2}}-\|z_1\|\big| \big|\frac{1}{\sqrt{2}}+\|z_1\|\big|=O(\delta).
\label{ex:sharp:1}
\end{align}
From \eqref{ex:sharp:1} it follows that
\begin{align}
\bigg|\sqrt{1-\|z_1\|^2}-\frac{1}{\sqrt{2}}\bigg|\lesssim \delta,    
\end{align}
and consequently, the term $\sqrt{1-\|z_1\|^2}\simeq \frac{1}{\sqrt{2}}.$ Moreover, by triangle inequality we obtain
$$\|z-\sqrt{2}\|z\|z_{1}\|\lesssim \|\hat{z}-\sqrt{2}z_{1}\|\lesssim \delta.$$
Now from the group law, it immediately implies that $f(x. \sqrt{2}\|z\| (z_1, 0)^{-1})=1,$ as $|t|<\delta.$ Similarly, we can show that $$\|z-\sqrt{2}\|z\| \sqrt{1-\|z_1\|^2} z_{2}\|=O(\|\hat{z}-z_2\|)\lesssim \delta,$$
 and hence $g(x. \sqrt{2}\|z\|((\sqrt{1-\|z_1\|^2} z_{2}, 0)^{-1})=1.$ Combining all these, we obtain that $$\mathfrak{M}_{\text{loc}}(f, g)(z, t)\geq \mathfrak{S}_{\sqrt{2}\|z\|}(f, g)(z, t)\geq |E^1_{(z, t)}|\, \sigma_{2n-1}(E^2_{(z, t)})\simeq \delta^{4n-1}.$$ 

As a consequence, we obtain \begin{align}
\delta^{4n-1}\delta^{1/p}\lesssim \delta^{4n-1}|\mf{R}_{\delta}|^{1/p}\lesssim C \delta^{(2n+1)(\frac{1}{p_1}+\frac{1}{p_2})}.  
\end{align}
Letting $\delta\to 0^+,$ we obtain
$$(4n-1)+\frac{1}{p}-(2n+1)(\frac{1}{p_1}+\frac{1}{p_2})\geq 0\iff \frac{1}{p_1}+\frac{1}{p_2}\leq \frac{4n-1}{2n+1}+\frac{1}{p(2n+1)}.$$
This completes the proof.
\end{proof}

We finally record the following necessary condition for the boundedness of $\mathfrak{M}_{\text{loc}}.$ 
\begin{prop}
\label{Nece:prop:2}
Let $1\leq p_1, p_2\leq \infty$ and $0<p<\infty.$ If we have
\[\|\mathfrak{M}_{\rm{loc}}(f, g)\|_{L^p(\Ha)}\leq C \|f\|_{L^{p_{1}}(\Ha)} \|g\|_{L^{p_{2}}(\Ha)},\]
then we must have
\begin{align}
 \frac{2}{p_1}+\frac{2}{p_2}\leq 1+\frac{2n
+1}{p}.  
\end{align}
\end{prop}

\begin{proof}
Let $\delta>0$ be sufficiently small and $Q_{\delta}$ denote the set $$\{(z, t): \big|\frac{1}{\sqrt{2}}-\|z\|\big|<\delta, |t|<\delta\}.$$ Let $f=g=\chi_{Q_{\delta}}.$ By slicing
\begin{align*}
&\mathfrak{S}(f,g)(x)\\
     &=\int_{B^{2n}(0,1)}f(x.(z_1, 0)^{-1})\cdot\\
     &\int_{S^{2n-1}} g(x. (\sqrt{1-\|z_1\|^2} z_{2}, 0)^{-1})\,(1-\|z_1\|^2)^{n-1} d\sigma_{2n-1}(z_2) dz_{1}\\
     &\geq \int_{\frac{1}{\sqrt{2}}<|z_1|<\frac{1}{\sqrt{2}}+\delta}f(x.(z_1, 0)^{-1})\cdot\\
     &\int_{S^{2n-1}} g(x. (\sqrt{1-\|z_1\|^2} z_{2}, 0)^{-1})\,(1-\|z_1\|^2)^{n-1} d\sigma_{2n-1}(z_2) dz_{1}.
\end{align*}
Denote $R_{\delta}:=\{x=(z, t): \|z\|<\delta, |t|<\delta\}.$ Then for $x=(z, t)\in R_{\delta},$ we observe that
\begin{align*}
\big|\|z-z_1\|-\frac{1}{\sqrt{2}}\big|&=\max\{\|z-z_1\|-\frac{1}{\sqrt{2}}, \frac{1}{\sqrt{2}}-\|z-z_1\|\}\\
&\leq \max\{\|z\|+\|z_1\|-\frac{1}{\sqrt{2}}, \frac{1}{\sqrt{2}}-\|z_1\|+\|z\|\}\lesssim \delta,  
\end{align*}
since $\frac{1}{\sqrt{2}}<\|z_1\|<\frac{1}{\sqrt{2}}+\delta$ and $\|z\|<\delta.$ Further, $\big|t-\frac{1}{2}\Im(z \bar{z_1})\big|\lesssim \delta,$ as $\|z\|, \|t\|<\delta.$ Therefore, $x. (z_1, 0)^{-1}\in Q_{\delta},$ thus $f(x. (z_1, 0)^{-1})=1.$ Further, observing that for $\frac{1}{\sqrt{2}}<\|z_1\|<\frac{1}{\sqrt{2}}+\delta$ we have $\frac{1}{\sqrt{2}}-\delta\leq \sqrt{1-\|z_1\|^2}\leq \frac{1}{\sqrt{2}},$ now performing similar computations as above one can show that $$x. (\sqrt{1-\|z_1\|^2} z_{2}, 0)^{-1}\in Q_{\delta}$$ and hence $g(x. (\sqrt{1-\|z_1\|^2} z_{2}, 0)^{-1})=1,$ for all $z_{2}\in S^{2n-1}.$

Combining above we obtain $\mathfrak{M}_{\text{loc}}(f, g)(x)\geq \delta$ for $x\in R_{\delta}.$ Therefore,
\begin{align*}
\delta \delta^{\frac{2n
+1}{p}}\lesssim \delta |R_{\delta}|^{\frac{1}{p}}\leq \|\mathfrak{M}_{\text{loc}}(f, g)\|_{L^p(\Ha)}\leq C \|f\|_{L^{p_{1}}(\Ha)} \|g\|_{L^{p_{2}}(\Ha)}\leq C \delta^{\frac{2}{p_1}+\frac{2}{p_2}}.    
\end{align*}
Letting $\delta\to 0^{+},$ we obtain $$\frac{2}{p_1}+\frac{2}{p_2}\leq 1+\frac{2n
+1}{p}.$$    
\end{proof}

\subsection*{Acknowledgements} AG gratefully acknowledges the support by the Industrial Consultancy and Sponsored Research (IC \& SR), Indian Institute of Technology Madras for the New Faculty Initiation Grant RF25261459MANFIG009296.

\end{document}